\documentclass[12pt,a4paper,oneside]{article}
\usepackage{amsfonts, amssymb, amsmath, amsthm}

\makeatletter
\renewcommand{\mod}[1]{\allowbreak\mkern 12mu{\operator@font mod}\,\,#1}
\makeatother

\allowdisplaybreaks 
\usepackage{bbold} 
\usepackage{mathtools}
\usepackage[utf8]{inputenc} 
\usepackage{comment} 
\usepackage[usenames,dvipsnames,svgnames,table]{xcolor}
\usepackage[right=0.5in,left=0.5in,top=1in,bottom=1in]{geometry}

\usepackage{fancyhdr} 
\usepackage{epigraph} 
\usepackage{floatrow} 
\usepackage{makecell}

\usepackage{tikz} 
\usepackage{tkz-euclide} 
\usepackage{tikz-cd} 
\usepackage{tkz-graph}
\usetikzlibrary{arrows}
\usetikzlibrary{arrows.meta} 
\usetikzlibrary{bending} 
\usetikzlibrary{calc}
\usetikzlibrary{matrix}
\usetikzlibrary{intersections}
\usetikzlibrary{patterns}
\usetikzlibrary{math} 
\usetikzlibrary{shapes} 
\usetikzlibrary{shapes.geometric} 
\usetikzlibrary{through,backgrounds}

\usetikzlibrary{lindenmayersystems}
\pgfdeclarelindenmayersystem{cayley}{
  \rule{F -> F [ R [F] [+F] [-F] ]}
  \symbol{R}{
    \pgflsystemstep=0.5\pgflsystemstep
  }
}

\usepackage{pgfplots}
\usepgfplotslibrary{polar}
\pgfplotsset{compat=newest}
\usepackage{nicematrix} 
\usepackage{capt-of} 
\usepackage[tableposition=top]{caption}
\usepackage{subcaption}
\usepackage[shortlabels]{enumitem}
\usepackage[ruled,linesnumbered]{algorithm2e} 
\usepackage{titlesec} 
\usepackage{booktabs} 
\usepackage{multicol} 
\usepackage[hypertexnames=false]{hyperref} 
\makeatletter
\usepackage{etoolbox}
\patchcmd\Hy@EveryPageBoxHook{\Hy@EveryPageAnchor}{\Hy@hypertexnamestrue\Hy@EveryPageAnchor}{}{\fail}
\makeatother
\usepackage[nameinlink]{cleveref}

\usepackage[backend=biber,style=numeric,sorting=nty,hyperref=true,
maxbibnames=99]{biblatex}

\usepackage{xurl}
\hypersetup{breaklinks=true}

\usepackage{makeidx} 
\makeindex 

\usepackage{libertine} 
\usepackage[libertine]{newtxmath} 

\DeclareMathAlphabet{\mathcal}{OMS}{cmsy}{m}{n}

\definecolor{edited}{rgb}{0,0,0}

\definecolor{codegreen}{rgb}{0,0.6,0}
\definecolor{codegray}{rgb}{0.5,0.5,0.5}
\definecolor{codepurple}{rgb}{0.58,0,0.82}
\definecolor{backcolour}{rgb}{0.95,0.95,0.92}

\makeatletter%
\@ifclassloaded{article}%
{\fancyhf{}

}
\makeatother%

\makeatletter%
\@ifclassloaded{book}%
{\renewcommand\tableofcontents{%
\chapter*{\contentsname
\@mkboth{%
\MakeUppercase\contentsname}{}}
\@starttoc{toc}%
}}
\makeatother%
\makeatletter%
\@ifclassloaded{book}%
{\titleformat{\part}[display]{\centering\normalfont\sffamily\Huge\bfseries\color{black}}
{\partname \space \thepart}{20pt}{\Huge}}%
\makeatother%

\titleformat{\section}{\normalfont\sffamily\Large\bfseries\color{black}}
{\thesection}{1em}{}
\titleformat{\subsection}{\normalfont\sffamily\large\bfseries\color{black}}
{\thesubsection}{1em}{}
\titleformat{\subsubsection}
{\normalfont\sffamily\normalsize\bfseries\color{black}}
{\thesubsubsection}{1em}{}
\titleformat{\paragraph}{\normalfont\sffamily}
{\paragraph \theparagraph}{1em}{}

\makeatletter%
\@ifclassloaded{book}%
{\titleformat{\chapter}[frame]
{\normalfont}{\filright\enspace \@chapapp~\thechapter\enspace}
{8pt}{\Large\bfseries\sffamily\filcenter}
\titlespacing*{\chapter}
{0pt}{0pt}{20pt}}
\makeatother%

\usepackage{thmtools}

\makeatletter%
\@ifclassloaded{book}%
{\def\boxnumber{chapter}}%
{\def\boxnumber{section}}%
\makeatother%

\newtheoremstyle{notation}
{\topsep}   
{\topsep}   
{\normalfont}  
{0pt}       
{\scshape} 
{.}         
{5pt plus 1pt minus 1pt} 
{}          

\theoremstyle{notation}

\newtheorem{theorem}{Theorem}[\boxnumber]
\newtheorem{lemma}[theorem]{Lemma}
\newtheorem{corollary}[theorem]{Corollary}

\makeatletter%
\renewenvironment{proof}[1][\proofname]{%
 \par\pushQED{\qed}\normalfont%
 \topsep6\p@\@plus6\p@\relax%
 \trivlist\item[\hskip\labelsep\itshape\sffamily#1\@addpunct{.}]
 \ignorespaces%
}{%
 \popQED\endtrivlist\@endpefalse%
}%
\makeatother%

\hypersetup{
 colorlinks=true,
 linktoc=all,
 citecolor=black,
 filecolor=black,
 linkcolor=black, 
 urlcolor=black
}

\newcommand*{\fullref}[1]{\textsf{\hyperref[{#1}]{\autoref*{#1}}}}

\newcommand{\bF}{\mathbb{F}}

\newcommand{\cI}{\mathcal{I}}

\newcommand{\cL}{\mathcal{L}}

\newcommand{\cP}{\mathcal{P}}

\let\originalleft\left
\let\originalright\right
\renewcommand{\left}{\mathopen{}\mathclose\bgroup\originalleft}
\renewcommand{\right}{\aftergroup\egroup\originalright}

\newcommand{\paren}[1]{\left( #1 \right)}

\newcommand{\ang}[1]{\left\langle #1 \right\rangle}

\newcommand{\defi}{\vcentcolon=}

\newcommand{\idgp}{\vmathbb{1}}
\newcommand{\subgp}{\leq}

\newcommand{\nsubgp}{\trianglelefteq}

\newcommand{\soc}[1]{\textrm{soc}(#1)}

\newcommand{\Centraliser}[2]{\textbf{C}_{#1}\left(#2\right)}

\newcommand{\Centre}[1]{\textbf{Z}\left(#1\right)}

\newcommand{\Normaliser}[2]{\textbf{N}_{#1}(#2)}

\def\PSL{\textsf{PSL}}
\def\PSU{\textsf{PSU}}

\def\PGL{\textsf{PGL}}

\newcommand{\PGLn}[2]{\PGL_{#1}({#2})}
\newcommand{\PSLn}[2]{\PSL_{#1}({#2})}
\newcommand{\PSUn}[2]{\PSU_{#1}({#2})}

\newcommand{\Sz}[1]{\prescript{2}{}{\textrm{B}_2}\paren{#1}}
\newcommand{\Ree}[1]{\prescript{2}{}{\textrm{G}_2}\paren{#1}}
\newcommand{\LargeRee}[1]{\prescript{2}{}{\textrm{F}_4}\paren{#1}}

\def\Aut{\textsf{Aut}}

\renewcommand\subgp{\leqslant}

\usepackage{authblk}

\author[]{Vishnuram Arumugam}
\author[]{John Bamberg}
\author[]{Michael Giudici}
\affil[]{\small Center for the Mathematics of Symmetry and Computation,
  University of Western Australia, Perth 6009, Australia\\
\href{mailto:Vishnuram.Arumugam@uwa.edu.au}{Vishnuram.Arumugam@uwa.edu.au}, 
\href{mailto:John.Bamberg@uwa.edu.au}{John.Bamberg@uwa.edu.au} and
\href{mailto:Michael.Giudici@uwa.edu.au}{Michael.Giudici@uwa.edu.au}}

\title{\sc Low rank groups of Lie type acting point and line-primitively on 
finite generalised quadrangles}

\begin{document}
\hypersetup{linkcolor=WildStrawberry}
\pagenumbering{arabic}
\setcounter{page}{1}
\pagestyle{plain}

\maketitle
\vspace*{-0.5cm}

\abstract{
Suppose we have a finite thick 
generalised quadrangle whose automorphism group $G$ acts
primitively on both the set of points and the set of lines. Then $G$ must be
almost simple. In this paper, we show that $\soc{G}$ cannot be isomorphic to
$\Sz{2^{2m+1}}$ or $\Ree{3^{2m+1}}$ where $m$ is a positive integer.  
}

\section{Introduction}
A generalised polygon is a type of incidence structure introduced by Jacques
Tits (1959) 
\cite{tits} to realise groups of Lie type as symmetries (more precisely,
automorphism groups) of geometric objects. Let $n$ be a positive integer. 
A generalised $n$-gon is an incidence 
structure whose incidence graph is a bipartite graph with diameter $n$ and 
girth $2n$. We say that a generalised
$n$-gon has order $(s,t)$ if every line has $s+1$ points incident with it and
every point is incident with $t+1$ lines. Furthermore, a generalised $n$-gon 
is said to be thick if it has order $(s,t)$ where $s,t > 1$. Feit and Higman
showed in \cite{feit} that thick generalised $n$-gons exist if and only
if \textcolor{edited}{$n \in \{3,4,6,8\}$. Our focus from now will only be on thick generalised
polygons. Therefore, we refer to a thick generalised polygon as simply a 
generalised polygon.} The examples of 
generalised polygons that arise from 
groups of Lie type via the construction of Tits are called \emph{classical}. 
Since then, many non-classical examples of projective planes (generalised 
$3$-gons) and generalised quadrangles ($4$-gons) have been found 
\cite[Section 3.7]{maldy}. In 
the case of generalised hexagons ($6$-gons) and generalised octagons
($8$-gons), the only known examples are the classical ones. \\

Many attempts have been made to construct new
examples. These attempts involve introducing various symmetry conditions on
the automorphism groups and analysing the possible groups acting on
generalised polygons with these symmetry conditions. Buekenhout and Van
Maldeghem (1994) \cite{buek} 
showed that if a group acts distance-transitively on a generalised
$n$-gon ($n \geqslant 4)$, then in fact, it must act point-primitively on that 
generalised $n$-gon. 
We summarise the results on generalised hexagons and
octagons in \Cref{tab:genhex} and the results on generalised quadrangles in
\Cref{tab:genquad}. Here, we let $\Gamma$ be a
generalised $n$-gon (generalised hexagon or octagon in \Cref{tab:genhex} and
generalised quadrangle in \Cref{tab:genquad}) and $G \subgp \Aut(\Gamma)$. 
For a point $\alpha$ in $\Gamma$, we write $G_{\alpha}$ to denote the 
point-stabiliser of $\alpha$ in $G$. Also, $q$ is assumed to be a prime power. 
Finally, the column of assumptions refers to the action of $G$ on $\Gamma$.

\begin{table}[H]
\centering
\begin{tabular}{|l|l|l|}
\hline
\thead{Assumptions} & \thead{Conclusion} & \thead{Reference} \\
\hline
\makecell{Point-primitive, \\ line-primitive \\ and flag-transitive} & 
$G$ is almost simple of Lie type & \cite{schneider} \\
\hline
Point-primitive & $G$ is almost simple of Lie type & \cite{bamberg} \\
\hline
\makecell{Point-primitive and \\ 
$\soc{G} \cong \PSLn{n}{q}$ for $n \geqslant 2$} &
$G_{\alpha}$ acts irreducibly on $V = \bF_q^n$ & \cite{glasbyoct} \\
\hline
\makecell{Point-primitive and \\ 
$\soc{G} \cong \Sz{2^{2m+1}}$, \\ $\Ree{3^{2m+1}}$ or 
$\LargeRee{2^{2m+1}}$} & \makecell{$\soc{G} \cong \LargeRee{2^{2m+1}}$ and 
$\Gamma$ is \\ the classical generalised octagon or its dual} & \cite{morgan} \\
\hline
\end{tabular}
\caption{Summary of results on generalised hexagons and octagons.}
\label{tab:genhex}
\end{table}

\begin{table}[H]
\centering
\begin{tabular}{|l|l|l|}
\hline
\thead{Assumptions} & \thead{Conclusion} & \thead{Reference} \\
\hline
\makecell{Point-primitive and \\ line-primitive} & $G$ is almost simple &
\cite{bgmrs} \\
\hline
\makecell{Point-primitive \\ line-primitive \\ and flag-transitive} & 
$G$ is almost simple of Lie type & \cite{bgmrs} \\
\hline
\makecell{Point-primitive, \\ flag-transitive and \\ $\soc{G} \cong A_n$ with
$n \geqslant 5$} & \makecell{$G \subgp S_6$ and $\Gamma$ is the unique \\ 
generalised quadrangle of order $(2,2)$} & \cite{bgmrs} \\
\hline
Point-primitive & \makecell{$\soc{G}$ is not isomorphic to \\ a sporadic group} 
& \cite{bambergevans} \\
\hline
\makecell{Point-primitive and \\ 
$\soc{G} \cong \PSLn{2}{q}$ for $q \geqslant 4$} &
\makecell{$q = 9$ and \\ $\Gamma$ is the symplectic quadrangle $W(2)$} & 
\cite{fenglu} \\
\hline
\makecell{Point-primitive and \\ 
line-primitive} & \makecell{$\soc{G}$ is not isomorphic to $\PSUn{3}{q}$ \\ 
for $q \geqslant 3$} & \cite{luzhangzou} \\
\hline
\end{tabular}
\caption{Summary of results on generalised quadrangles.}
\label{tab:genquad}
\end{table}

The main theorem in this paper (\Cref{thm:main}) is motivated by the work of 
Morgan and Popiel \cite{morgan} on
generalised hexagons and octagons, where they showed that if $G$ acts
point-primitively on a generalised hexagon or an octagon $\Gamma$ with 
socle $\soc{G}
\cong \Sz{2^{2m+1}}$, $\Ree{3^{2m+1}}$ or $\LargeRee{2^{2m+1}}$, then $\soc{G}
\cong \LargeRee{2^{2m+1}}$ and $\Gamma$ is the classical generalised octagon
or its dual.

\begin{theorem}
\label{thm:main}
Let $m$ be a positive integer. An almost simple group with socle isomorphic to 
$\Sz{2^{2m+1}}$ or $\Ree{3^{2m+1}}$ cannot act primitively on both the set of
points and the set of lines of a generalised quadrangle.
\end{theorem}

The case where $\soc{G} \cong \LargeRee{2^{2m+1}}$ is not included in this
paper as it still requires further analysis and the techniques needed to 
classify generalised quadrangles in this scenario may be different.


\section{Preliminaries}
We develop some preliminary definitions and results from incidence geometry 
and group theory.

\subsection{Incidence Geometry}
An incidence geometry is a triple $(\cP,\cL,\cI)$ where $\cP$ is called the
set of points, $\cL$ is the set of lines disjoint from $\cP$ and 
$\cI \subseteq \cP \times \cI$ is the incidence relation. We say that a point 
$\alpha \in \cP$ is incident with a
line $L \in \cL$ if $(\alpha,L) \in \cI$. This pair $(\alpha,L)$ is called a
flag. We say two points $\alpha, \beta \in \cP$ are collinear if there exists 
a line $L \in \cL$ such that $(\alpha,L) \in \cI$ and $(\beta,L) \in \cI$. 
Finally, the dual of an incidence geometry $(\cP,\cL,\cI)$ is the
incidence geometry $(\cP^{\prime},\cL^{\prime},\cI^{\prime})$ where
the set of points $\cP^{\prime} = \cL$, the set of lines $\cL^{\prime} = \cP$ 
and the incident pair $(L,\alpha) \in \cI^{\prime}$ precisely when 
$(\alpha,L) \in \cI$. \\ 

We may also view an incidence structure as a graph. More precisely, given an
incidence structure $\Gamma = (\cP,\cL,\cI)$, we construct a graph called the
\emph{incidence graph} of $\Gamma$ in the following way: we take the vertex
set to be $\cP \cup \cL$ and we join an edge between $\alpha \in \cP$ and 
$L \in \cL$ if $(\alpha,L) \in \cI$. Note that the disjointness of $\cP$ 
and $\cL$ ensures that the incidence graph does not contain loops. Finally, an
automorphism of $\Gamma$ is a bijective map $\theta: \cP \cup \cL \rightarrow 
\cP \cup \cL$ that sends points to points and lines to lines as well as 
preserving the incidence relation. 
The group of automorphisms of $\Gamma$ is denoted by $\Aut(\Gamma)$. 
While we have defined generalised polygons via this incidence graph, we will
often think of them as incidence geometries. For more
information on incidence structures and generalised polygons, we refer to
\cite{kiss,payne,maldy}. \\

From here on, a generalised quadrangle is assumed to be finite, \emph{i.e.},
it has a finite set of points and set of lines. Furthermore, a generalised
quadrangle with order $(s,t)$ has $(s+1)(st+1)$ points and $(t+1)(st+1)$
lines \cite[(1.2.1)]{payne}.

\begin{lemma}[{\cite[Corollary 2.3]{fenglu}}]
\label{lem:suzsub}
Let $G$ be a group acting on a generalised quadrangle $\Gamma$ and suppose $g
\in G$. 
Let $\Gamma_g = (\cP_g,\cL_g,\cI_g)$ be the fixed substructure of $g$. If 
$|\cP_g| \geqslant 2$, $|\cL_g| \geqslant 2$ and $\Gamma_g$ admits an
automorphism group $H$ that is transitive on both points and lines, then
$\Gamma_g$ is a generalised quadrangle of order $(s^{\prime},t^{\prime})$ for
some positive integers $s^{\prime}$ and $t^{\prime}$.
\end{lemma}

The next number-theoretic lemma concerns the solutions of $s$ and $t$ given a
certain number of points and lines.  
\begin{lemma}
\label{lem:suzupts}
Let $a,b,s,t$ be positive integers and $p$ be a prime. If  
\begin{equation}
(s+1)(st+1) = p^a \ \textrm{and} \ (t+1)(st+1) = p^b,
\label{eqn:suzupts}
\end{equation}
then $p = 2$. Furthermore, we have: $(s = 1 \ \textrm{and} \ b = a/2+1)$ or
$(t = 1 \ \textrm{and} \ a = b/2+1)$.
\end{lemma}
\begin{proof}
\textcolor{edited}{Let $d \defi \gcd(s+1,t+1,st+1)$. Then $d$ divides $t(s+1)-(st+1) = t-1$.
Therefore, $d$ divides $t+1-(t-1) = 2$. Hence, $d = 1$ or $d = 2$. Since 
$s+1,t+1,st+1 \geqslant 2$, it follows that $p$ divides $s+1,t+1$ and $st+1$ 
and so, $p$ divides $d$. From which, we deduce that $d = 2$ and $p = 2$. Now,
suppose that $s+1 \neq 2$ and $t+1 \neq 2$. Then $s+1 \equiv 0 \pmod 4$ and
$t+1 \equiv 0 \pmod 4$. Thus, $st+1 \equiv (-1)(-1)+1 \equiv 2 \pmod 4$ and so, $st+1=2$. 
This is a contradiction since $st+1 > s+1 > 2$.} Therefore, $s+1 = 2$
or $t+1 = 2$. Without loss of generality, 
say $t = 1$. Substituting for $t$ and $p$ in \eqref{eqn:suzupts}, we obtain 
\[
(s+1)^2 = 2^a \ \textrm{and} \ 2(s+1) = 2^b.
\]
Hence, $s+1 = 2^{a/2} = 2^{b-1}$, from which, we obtain $b = a/2+1$. If we 
consider the case where $s = 1$, then by the same argument, we obtain $a =
b/2+1$. 
\end{proof}

\subsection{Group Theory}
\subsubsection{Notation}
Let $n$ be a positive integer and $q$ be a prime power. 
We denote the cyclic group of order $n$ by $C_n$, the dihedral group of order
$2n$ by $D_n$, the elementary abelian group of order $q$ by $E_q$. Given two
groups $H$ and $K$, we write $H.K$ to mean an extension of $H$ by $K$,
\emph{i.e.}, a group $G$ with a normal subgroup $M$ where $M \cong H$ and $G/M
\cong K$. When the extension is a split extension, we write $H:K$ (or $H
\rtimes K$). 

\subsubsection{Elementary Results}
In our study of groups of Lie type, we will come across subgroups that are
semidirect products of two groups with coprime order. A classical result that
will be useful in studying these subgroups is the Schur-Zassenhaus Theorem,
which can be found in \cite[Section 3B]{isaacsfin}.
 
\begin{theorem}[Schur-Zassenhaus Theorem]
Let $G$ be a group with a normal subgroup $N$ such that $|N|$ and $|G/N|$ are
coprime.
Then there exists a subgroup $K \subgp G$ such that $G = N \rtimes K$,
\emph{i.e.}, $G = NK$ and $N \cap K = \idgp$. We call $K$ a \emph{complement}
of $N$ in $G$. Moreover, all complements of $N$ in $G$ are conjugate to $K$.  
\end{theorem}

\begin{corollary}
\label{lem:schur}
Let $G = N \rtimes K$ where $\gcd(|N|,|K|) = 1$. Suppose $M \subgp G$
with $|M| = |K|$. Then $M = K^x$ for some $x \in G$. 
\end{corollary}
\begin{proof}
Observe that $N \cap M = \idgp$ since they have coprime orders. Thus,
$|NM| = |N||M|/|N \cap M| = |N||K| = |G|$. Hence,
$M$ is a complement of $N$ in $G$. By the Schur-Zassenhaus Theorem,
all complements of $N$ are conjugate and so, $M = K^x$ for some $x \in G$. 
\end{proof}

\begin{corollary}
\label{cor:conj}
Let $G = N \rtimes K$ where $\gcd(|N|,|K|) = 1$ and suppose $g \in G$ such
that $|g|$ divides $|K|$. Then $g \in K^h$ for some $h \in G$. In particular,
the number of conjugacy classes of elements with order $|g|$ in $G$ is at most 
the number of conjugacy classes of elements with order $|g|$ in $K$.
\end{corollary}
\begin{proof}
First, observe that $g$ acts on $N$ by conjugation. 
Since $|g|$ divides $|K|$, it follows that $g \notin N$. 
Hence, we have the subgroup 
$M \defi N \rtimes \ang{g} \subgp G$. Note that $KM =
K(N\rtimes \ang{g}) = KN\ang{g} = G$ and $|G| = |K||M|/|K \cap M|$. Thus, 
\[
|K \cap M| = \frac{|K||M|}{|G|} = \frac{|K||N||\ang{g}|}{|G|} =
\frac{|K||N||g|}{|K||N|} = |g|.
\]
By \Cref{lem:schur}, we deduce that 
$K \cap M$ is a complement of $N$ in $M$. By the
Schur-Zassenhaus Theorem, all complements of $N$ in $M$ are conjugate.
Therefore, $K \cap M = \ang{g}^x$ for some $x \in M$. 
Thus, $\ang{g} \subgp K^{x^{-1}}$, whence, $g \in K^{h}$, where $h = x^{-1}$.
The number of conjugacy classes of elements with order $|g|$ is at most the
number in $K$ because any element of $G$ with order $|g|$ is conjugate to an
element of $K$. 
\end{proof}

\subsubsection{Permutation Groups}
Our investigation involves groups of Lie type acting on generalised
quadrangles. The following lemma provides a formula for calculating the
number of fixed points of an element with respect to a group action. 
\begin{lemma}[Formula for the Number of Fixed Points 
{\cite[Lemma 2.5]{fixpts}}]
\label{lem:fixpts}
Let $G$ be a finite group acting transitively on a set $\Omega$. Let $\alpha
\in \Omega$ and $g \in G$. Then the number of fixed points of $g$, denoted 
$\pi(g)$, is given by
\begin{equation}
\pi(g) = \frac{|\Omega| |g^G \cap G_{\alpha}|}{|g^G|}.
\label{eqn:fixpts}
\end{equation}
\end{lemma}

\textcolor{edited}{
\begin{lemma}[{\cite[Lemma 2.4]{luzhangzou}}]
\label{lem:cgtranss}
Let $G$ be a group acting transitively on a set $\Omega$ and $g$ be a
non-identity element in $G$. Consider a point $\alpha \in \Omega_g$, where
$\Omega_g$ is the set of fixed points of $g$. Then $\Centraliser{G}{g}$ acts
transitively on $\Omega_g$ if and only if $g^G \cap G_{\alpha} =
g^{G_{\alpha}}$, \emph{i.e.}, the conjugacy class of $g$ in $G$ does not split
into multiple conjugacy classes in $G_{\alpha}$. In this case, $|\Omega_g| =
|\Centraliser{G}{g} : \Centraliser{G}{g} \cap G_{\alpha}|$. 
\end{lemma}}


\section{Suzuki Groups}
\subsection{Structure of Suzuki Groups}
In this subsection, we provide some background information on the family of 
Suzuki groups. Throughout this section, we denote a Suzuki group by $\Sz{q}$, 
where $q = 2^{2m+1}$ for some positive integer $m$. We start with the
following result which can be found in \cite[Theorem 9]{suzuki}. 
\begin{theorem}[Maximal Subgroups of a Suzuki Group]
\label{thm:suzukigroups}
Let $G = \Sz{q}$ and $H$ be a maximal subgroup of $G$. Then $H$ is isomorphic
to one of the following:
\begin{enumerate}[(i)]
\item $M \cong E_q.E_q.C_{q-1}$;
\item $B_0 = \Normaliser{G}{A_0} \cong D_{q-1}$, where $A_0 \cong C_{q-1}$;
\item $A_1 \cong C_{q+\sqrt{2q}+1}$, $A_2 \cong C_{q-\sqrt{2q}+1}$;
\item $B_i = \Normaliser{G}{A_i} \cong C_{q\pm\sqrt{2q}+1} : C_4$ for $i \in
\{1,2\}$; and 
\item $N_0 \cong \Sz{q_0}$, where $q_0 = 2^{n_0} > 2$ and $q = q_0^{r_0}$ for
some prime $r_0$.
\end{enumerate}
Moreover, there is only one conjugacy class of each type of maximal subgroup.
\end{theorem}

The next lemma shows that all the involutions in a Suzuki group are conjugate.
\begin{lemma}[Conjugacy Class of Involutions]
\label{lem:suzuinvs}
Let $G = \Sz{q}$ and $H$ be a maximal subgroup of $G$. Then $G$ and $H$ both 
have exactly one conjugacy class of involutions.
\end{lemma}
\begin{proof}
It was shown in \cite[Proposition 7]{suzuki} that all the involutions are 
conjugate in $G$. Furthermore, if $H \cong
E_{q}.E_{q}.C_{q-1}$, then a cyclic subgroup of $H$ of order $q-1$ permutes 
the set of involutions in $H$ transitively \cite[Proposition 8]{suzuki}. 
Therefore, 
we only need to consider the cases:\footnote{Note that the subfield case: 
$H \cong \Sz{q_0}$, we have that $H$ is a Suzuki group and so, it follows from 
\cite[Proposition 7]{suzuki}.} 
$H \cong D_{q-1} = C_{q-1}:C_2$ and 
$H \cong C_{q\pm\sqrt{2q}+1}:C_4$. 
Since $q$ is a power of $2$, we have that $\gcd(q-1,2) = 1$ and
$\gcd(q\pm\sqrt{2q}+1,4) = 1$. Hence, by \Cref{cor:conj}, we conclude 
that all the involutions are conjugate in $G$. 
\end{proof}

The following lemma provides information about the centraliser of an
involution in $G$. 
\begin{lemma}[Centraliser of an Involution]
\label{lem:suzucentinv}
Let $G = \Sz{q}$ with maximal subgroups $H,K \subgp G$ where $H \cong D_{q-1}$
and $K \cong C_{q\pm\sqrt{2q}+1}:C_4$. Let $g \in G$, $h \in H$ and $k \in K$
be involutions. Then 
\[
|\Centraliser{G}{g}| = q^2, \ |\Centraliser{H}{h}| = 2, \
\textrm{and} \ |\Centraliser{K}{k}| = 4.
\]
\end{lemma}
\begin{proof}
Consider $G = \Sz{q}$. Note that $\Centraliser{G}{g} \subgp
Q$ where $Q$ is a Sylow $2$-subgroup of order $q^2$ 
\cite[Proposition 1]{suzuki}. Moreover, $g \in \Centre{Q}$ 
\cite[Proposition 7]{suzuki}. Therefore, $|\Centraliser{G}{g}| = |Q| = q^2$. 
\\

Next, we find $\Centraliser{H}{h} = \Centraliser{G}{h} \cap H$. Observe that 
$|\Centraliser{H}{h}| = |\Centraliser{G}{h} \cap H| \leqslant 2$ since
$|\Centraliser{G}{h}| = q^2$ and $|H| = 2(q-1)$. Note that $h \in
\Centraliser{H}{h}$. Hence, $\Centraliser{H}{h} \cong C_2$. Similarly, we have 
$|\Centraliser{K}{k}| = |\Centraliser{G}{k} \cap K| \leqslant 4$ since
$|\Centraliser{G}{k}| = q^2$ and $|K| = 4(q\pm\sqrt{2q}+1)$. Let us now write 
$K = \ang{x}:\ang{y}$, where $|x| = q\pm\sqrt{2q}+1$ and $|y| = 4$. Using 
\Cref{cor:conj}, we find that $k = (y^2)^{z}$ for some $z \in K$, whence,
$y^z \in \Centraliser{K}{k}$. Since $|y| = 4$, we have that
$\Centraliser{K}{k} \cong C_4$. 
\end{proof}

\begin{lemma}[Number of Fixed Points of an Involution]
\label{lem:fixptsinv}
Let $G = \Sz{q}$ act primitively on a set $\Omega$ and take a point $\alpha
\in \Omega$. Suppose $g \in G_{\alpha}$ is an involution. Then 
\[
|\Omega_g| = \begin{cases}
q^2/2 = 2^{4m+1} & \textrm{if $G_{\alpha} \cong D_{q-1}$}, \\
q^2/4 = 2^{4m} & \textrm{if $G_{\alpha} \cong C_{q\pm\sqrt{2q}+1}:C_4$}, \\
(q/q_0)^2 = 2^{2n_0(r_0-1)} & \textrm{if $G_{\alpha} \cong \Sz{q_0}$ where
$q_0 = 2^{n_0} > 2$} \\ 
& \quad \textrm{and $q
= q_0^{r_0}$ for some prime $r_0$.}
\end{cases}
\]
Furthermore, $\Centraliser{G}{g}$ acts transitively on $\Omega_g$, the set of 
points fixed by $g$.
\end{lemma}
\begin{proof}
Since $G_{\alpha}$ contains a single class of involutions
(\Cref{lem:suzuinvs}), we conclude that $g^G \cap G_{\alpha} =
g^{G_{\alpha}}$. Thus, applying \Cref{lem:cgtranss}, we find that
$\Centraliser{G}{g}$ acts transitively on $\Omega_g$ and $|\Omega_g| =
|\Centraliser{G}{g}: \Centraliser{G}{g} \cap G_{\alpha}|$. 
Using \Cref{lem:suzucentinv}, we obtain the desired formulae for
$|\Omega_g|$. 
\end{proof}

\subsection{Suzuki Groups Acting on Generalised Quadrangles}
We now prove \Cref{thm:main} in the case where the socle is isomorphic to 
$\Sz{q}$.

\begin{proof}
Suppose that an almost simple group with socle $\Sz{q}$ acts primitively on 
the point-set and the line-set of a generalised quadrangle $\Gamma = 
(\cP,\cL,\cI)$ with order $(s,t)$ where $s,t \geqslant 2$. 
Then $|\cP| = (s+1)(st+1)$. Note that $s+1 \geqslant 3$ and 
$st+1 \geqslant 3$. Therefore, their product cannot be a prime. The argument 
in \cite[Section 3]{morgan} shows that the socle of this almost simple group 
acts primitively on both $\cP$ and $\cL$. Therefore, it suffices to consider 
the Suzuki group $\Sz{q}$ acting primitively on both $\cP$ and $\cL$. Let 
$G = \Sz{q}$. For a point $\alpha \in \cP$ 
and a line $L \in \cL$, their respective stabilisers $G_{\alpha}$ and $G_{L}$ 
are maximal subgroups of $G$. By \Cref{thm:suzukigroups}, a maximal subgroup 
of $\Sz{q}$ is isomorphic to one of the following:
\begin{enumerate}[(i)]
\item (Parabolic) $E_q . E_q . C_{q-1}$;
\item (Dihedral) $D_{q-1}$;
\item (Frobenius) $C_{q\pm \sqrt{2q} + 1} : C_4$; or 
\item (Subfield) $\Sz{q_0}$, where $q_0 = 2^{n_0} > 2$ and $q = q_0^{r_0}$ for 
some prime $r_0$. 
\end{enumerate}
Moreover, there is only one conjugacy class of each type of maximal subgroup, 
see also \cite[Table 8.16]{maxlowdim}. \\

Note that if $G_\alpha \cong E_q.E_q.C_{q-1}$, then by \cite{suzuki}, the 
action of $G$ is $2$-transitive. Thus, a pair of collinear points can
be mapped to a pair of non-collinear points, which is a contradiction. 
Therefore, we only need to investigate the cases: $D_{q-1}$, $C_{q\pm
\sqrt{2q}+1}:C_4$ and $\Sz{q_0}$ for the point and line-stabiliser and show
that there are no generalised quadrangles in those scenarios. To this end,
suppose $G_{\alpha}$ is isomorphic to $D_{q-1}$, $C_{q\pm\sqrt{2q}+1}:C_4$ or
$\Sz{q_0}$ and $G_{L}$ is isomorphic to $D_{q-1}$, $C_{q\pm\sqrt{2q}+1}:C_4$
or $\Sz{q_1}$, where $q_i = 2^{n_i} > 2$ and $q_i^{r_i} = q$ for primes 
$r_i$ and $i = 1,2$. 
Observe that both $G_{\alpha}$ and $G_L$ contain an involution,
say $g \in G_{\alpha}$ and $h \in G_L$. Since all the involutions in $G$ are
conjugate (\Cref{lem:suzuinvs}), there exists a $k \in G$ such that $g =
h^k$. Thus, $g \in G_L^k$. Since $G_L^k = G_{L^{\prime}}$ where 
$L^{\prime} = L^k \in \cL$, we can take $L$ to be $L^{\prime}$ and assume that 
we have an involution $g \in G_{\alpha} \cap G_L$. Let $\cP_g$ and $\cL_g$ be 
the set of points and the set of lines fixed by $g$, respectively. 
By \Cref{lem:fixptsinv}, we find that the number of fixed points and lines are 
\[
|\cP_g| = 2^a \ \textrm{and} \ |\cL_g| = 2^b,
\]
where $a \in A \defi \{4m,4m+1,2n_0(r_0-1)\}$ and $b \in
B \defi \{4m,4m+1,2n_1(r_1-1)\}$. 
By \Cref{lem:suzsub}, the
fixed substructure, $(\cP_g,\cL_g,\cI_g)$ is a generalised quadrangle of
order $(s^{\prime},t^{\prime})$ for some positive integers $s^{\prime}$ and
$t^{\prime}$. Therefore, 
\[
|\cP_g| = (s^{\prime}+1)(s^{\prime}t^{\prime}+1) = 2^a \ \textrm{and} \ 
|\cL_g| = (t^{\prime}+1)(s^{\prime}t^{\prime}+1) = 2^b.
\]
By \Cref{lem:suzupts}, we have that $b = a/2+1$ or dually, $a = b/2+1$. Let us
suppose that $b = a/2+1$. Thus, $a$ is even and so, $a = 4m$ or $2n_0(r_0-1)$. 
If $a = 4m$, then $b-1 = a/2 = 2m$. 
Hence, $b = 2m+1$, which is odd. However, the
only odd element in $B$ is $4m+1$ and so, we have a contradiction. Next, 
if $a = 2n_0(r_0-1)$, then $b = a/2+1 = n_0(r_0-1)+1$. 
However, since $r_0$ is an odd
prime, it follows that $b$ is odd. Therefore, $b = 4m+1$. Consequently, we
find $4m+1 = b = n_0r_0-n_0+1 < n_0r_0 = 2m+1$, which is a contradiction. The
case where $a = b/2+1$ is analogous. Therefore, the Suzuki 
group $\Sz{q}$ cannot act primitively on both the set of points and the set of 
lines of a generalised quadrangle. Consequently, an almost simple group with 
socle isomorphic to $\Sz{q}$ cannot act primitively on both the set of 
points and the set of lines of a generalised quadrangle.  
\end{proof}


\section{Ree Groups}
We provide some background information on the family of small Ree
groups. Throughout this section, we denote a small Ree group by $\Ree{q}$, 
where $q = 3^{2m+1}$ for some positive integer $m$. 

\subsection{Structure of Ree Groups}
We have the following theorem regarding the maximal subgroups of $\Ree{q}$
(see \cite[Table \textcolor{edited}{8.43}]{maxlowdim}). 
\begin{theorem}[Maximal Subgroups of a Small Ree Group]
\label{thm:reemaxs}
Let $G = \Ree{q}$ and $H$ be a maximal subgroup of $G$. Then \textcolor{edited}{$H$ is 
isomorphic to} one of the following: 
\begin{enumerate}[(i)]
\item $E_q.E_q.E_q.C_{q-1}$;
\item $C_2 \times \PSLn{2}{q}$;
\item $(E_4 \times D_{(q+1)/4}) : C_3$;
\item $C_{q\pm\sqrt{3q}+1}:C_6$; and 
\item $\Ree{q_0}$, where $q_0 = 3^{n_0} > 3$ and $q = q_0^{r_0}$ for some 
prime $r_0$. 
\end{enumerate}
Moreover, there is only one conjugacy class of each type of maximal subgroup.
\end{theorem}
Analogous to the Suzuki groups, elements of order $3$ play a crucial role in 
studying the action of a small Ree group on a generalised quadrangle. First, 
we recall a definition from group theory. 
For a group $G$ and $g \in G$, we say that $g$ is \textit{real} in $G$ if
$g$ is conjugate to its inverse in $G$, \emph{i.e.}, there exists an element 
$h \in G$ such that $g^{-1} = g^h$.
We may omit the ``in $G$'' part and simply refer to $g$ as a real element. \\

The following lemma is useful for analysing the maximal subgroup isomorphic to
$(E_4 \times D_{(q+1)/4}) : C_3$. 
\begin{lemma}
\label{lem:ycentinv}
Let $H = KR \cong (E_4 \times D_{(q+1)/4}) : C_3$ where $K \cong E_4 \times
D_{(q+1)/4}$ and $R = \ang{y} \cong C_3$. Then $y$ centralises an involution
in $H$. 
\end{lemma}
\begin{proof}
Note that $\ang{y}$ acts by conjugation on the 
set of involutions in $H$. We count the number of involutions in $H$. 
Since $q+1 = 3^{2m+1}+1 = 3(3^2)^m+1 \equiv 3+1 \pmod 8 \equiv 4 \pmod 8$, it 
follows that $(q+1)/4$ is odd and thus,
$D_{(q+1)/4}$ has $(q+1)/4$ involutions. 
Since an involution in $K$ is the product of an element of $E_4$ and an
involution in $D_{(q+1)/4}$, or is just an involution in $E_4$, the number of
involutions in $K$ is $4(q+1)/4+3 = q+4$.
Now, all the involutions in $H$ are in $K$ because $K \nsubgp H$ and
$\gcd(|K|,|R|) = 1$. Therefore, $H$ has $q+4$ involutions. Focusing on the
action of $\ang{y}$ on the set of involutions, suppose that $\ang{y}$ has no
fixed points, \emph{i.e.}, $y$ does not centralise any involution. Then the
orbits of $\ang{y}$ must have length divisible by $3$. 
This implies that $q+4$ is 
divisible by $3$, which is a contradiction as $q = 3^{2m+1}$. Therefore, $y$ 
has a fixed point, \emph{i.e.}, $y$ centralises an involution. 
\end{proof}

\begin{lemma}[Conjugacy Classes of Elements of Order Three]
\label{lem:reeccls}
Let $G = \Ree{q}$ and $H$ be a maximal subgroup of $G$. Then the following
statements hold:
\begin{enumerate}[(i)]
\item In $G$, there is one conjugacy class of real elements of 
order $3$ and two conjugacy classes of non-real elements of order $3$. 

\item If $H \cong C_2 \times \PSLn{2}{q}$, 
$H \cong (E_4 \times D_{(q+1)/4}): C_3$, or $H \cong C_{q\pm\sqrt{3q}+1}:C_6$, 
then there are two $H$-conjugacy classes of non-real elements of order 
$3$.
\end{enumerate}
\end{lemma}
\begin{proof}
First, we find that the centre of a
Sylow $3$-subgroup of $G$ contains one conjugacy class of elements with order
$3$ with representative labelled $X$ \cite[Chapter III, Paragraph 4]{ward}. 
Also, in \cite[Chapter III, Paragraph 7]{ward}, we find that there are two 
conjugacy classes of elements of order $3$ with representatives labelled $T$ 
and $T^{-1}$. This yields (i). \\

Now we focus on (ii). From \cite[Chapter III, Paragraphs 1 and 3]{ward}, we
deduce that an order $3$ element in $G$ is not real precisely when it 
centralises an involution. Therefore, it suffices to show the following two
points: 
\begin{enumerate}[(a)]
\item An element of order $3$ in $H$ centralises an involution.
\item There are at most two conjugacy classes of elements with order $3$ in 
$H$.
\end{enumerate}
Indeed, once we have established that there are at most two conjugacy classes, 
it follows that we have exactly two conjugacy classes since the elements are 
not real. \\

Let us suppose $H \cong C_2 \times \PSLn{2}{q}$. Note that all the
order $3$ elements of $H$ are in $\PSLn{2}{q}$. Therefore, the involution in
$C_2$ is centralised by the order $3$ elements of $H$. Also, there are exactly
two conjugacy classes of elements with order $3$ in $\PSLn{2}{q}$
(see \cite[Chapter 3, (6.3) (iii)]{psl}) and hence, two conjugacy classes in 
$H$. \\

Next, we consider $H = KR \cong (E_4 \times D_{(q+1)/4}):C_3$ where $K \cong
E_4 \times D_{(q+1)/4}$ and $R = \ang{y} \cong C_3$. Observe that 
$\gcd(|E_4 \times D_{(q+1)/4}|, |C_3|) =
\gcd(2(q+1),3) = 1$. Hence, by \Cref{cor:conj}, any element of order $3$ must
be conjugate in $H$ to $y$ or $y^{-1}$. So there are at most two conjugacy
classes of elements with order $3$. Furthermore, $y$ centralises an involution
in $H$ by \Cref{lem:ycentinv}. \\
 
Finally, let us suppose that $H = KR \cong C_{q\pm\sqrt{3q}+1}:C_6$ where 
$K \cong
C_{q\pm\sqrt{3q}+1}$ and $R = \ang{y} \cong C_6$. We find that 
$\gcd(|C_{q\pm\sqrt{3q}+1}|,|C_6|) = \gcd(q\pm\sqrt{3q}+1,6) = 1$ since 
$q\pm\sqrt{3q}+1 \equiv 1\pm1+1 \pmod 2
\equiv 1 \pmod 2$ and $q\pm\sqrt{3q}+1 \equiv 1 \pmod 3$. There are two
conjugacy classes of elements with order $3$ in $R$, namely, $(y^2)^R$ and
$(y^4)^R$. Using
\Cref{cor:conj} again, we find that the number of conjugacy classes of 
elements with order $3$ in $H$ is at most $2$. Moreover, we note that $y^2$
centralises $y^3$, which is an involution. 
\end{proof}
 
\begin{lemma}[Centralisers of Elements of Order Three]
\label{lem:reecent}
Let $G = \Ree{q}$, $x \in G$ be a real element
of order $3$ and $y \in G$ be a non-real element of order $3$. Then 
\[
|\Centraliser{G}{x}| = q^3 \ \textrm{and} \ |\Centraliser{G}{y}| = 2q^2.
\]
Consider maximal subgroups $H_1$, $H_2$ and $H_3$ of $G$ where 
$H_1 \cong C_2 \times \PSLn{2}{q}$, $H_2 \cong (E_4 \times D_{(q+1)/4}):C_3$ 
and $H_3 \cong C_{q\pm\sqrt{3q}+1}:C_6$. Suppose $h_i \in H_i$ are elements of
order $3$ for $i \in \{1,2,3\}$. Then 
\[
|\Centraliser{H_1}{h_1}| = 2q, \ |\Centraliser{H_2}{h_2}| = 6 \
\textrm{and} \ |\Centraliser{H_3}{h_3}| = 6.
\]
\end{lemma}
\begin{proof}
From \cite[Chapter III, Paragraphs 2 and 3]{ward}, we obtain
$|\Centraliser{G}{y}| = 2q^2$ and $|\Centraliser{G}{x}| = q^3$, 
respectively. \\

From \cite[Chapter 3, (6.4) (i)]{psl}, 
we find that $|\Centraliser{\PGLn{2}{q}}{x}| 
= q$ when $x$ is an element of order $3$. Since $\PSLn{2}{q}$ has index $2$ in
$\PGLn{2}{q}$ and $C \defi \Centraliser{\PGLn{2}{q}}{x}$ has odd order, it
follows that $C \subgp \PSLn{2}{q}$. Now, back to $H_1$. Since the $C_2$ in 
$H_1$ also centralises $h_1$, we obtain $|\Centraliser{H_1}{h_1}| = 2q$. \\

Next, observe that $|\Centraliser{H_2}{h_2}| =
|\Centraliser{G}{h_2} \cap H_2| \leqslant 6$ since
$\gcd(|\Centraliser{G}{h_2}|,|H_2|) = \gcd(2q^2,4 \cdot 2(q+1)/4 \cdot 6) =
\gcd(2q^2,12(q+1)) = 6$. Furthermore, $h_2$ centralises an involution in $H_2$
(\Cref{lem:ycentinv}). Therefore, $|\Centraliser{H_2}{h_2}| = 6$. \\

Finally, we consider $H_3$. By the argument as above, we see that
$\gcd(|\Centraliser{G}{h_3}|,|H_3|) = \gcd(2q^2,6(q\pm\sqrt{3q}+1)) = 6$,
whence,
$|\Centraliser{H_3}{h_3}| = |\Centraliser{G}{h_3} \cap H_3| \leqslant 6$. 
Since $h_3$ lies in some subgroup isomorphic to $C_6$ in $H_3$ 
(\Cref{cor:conj}), it follows that $h_3$ is
centralised by $6$ elements. Therefore, $|\Centraliser{H_3}{h_3}| = 6$. 
\end{proof}

\begin{lemma}[Number of Fixed Points of Order Three Elements]
\label{lem:reefixpts}
Let $G = \Ree{q}$ act primitively on a set $\Omega$ and take a point $\alpha
\in \Omega$. Suppose we have a non-real element $g \in G_{\alpha}$ of order
$3$. Then
\[
|\Omega_g| = \begin{cases}
q & \textrm{if $G_{\alpha} \cong C_2 \times \PSLn{2}{q}$}, \\
q^2/3 & \textrm{if $G_{\alpha} \cong (E_4 \times D_{(q+1)/4}):C_3$ or
$G_{\alpha} \cong C_{q\pm\sqrt{3q}+1}:C_6$} \\
(q/q_0)^2 & \textrm{if $G_{\alpha} \cong \Ree{q_0}$ where $q_0 = 3^{n_0} > 3$
and $q = q_0^{r_0}$ for some prime $r_0$}.
\end{cases}
\]
Furthermore, $\Centraliser{G}{g}$ acts transitively on $\Omega_g$, the set of 
points fixed by $g$. 
\end{lemma}
\begin{proof}
By \Cref{lem:reeccls}, $G_{\alpha}$ has precisely two conjugacy classes
of order $3$ non-real elements. Since $G$ also has two conjugacy classes of
these elements, the conjugacy classes do not split, \emph{i.e.}, $g^{G} \cap
G_{\alpha} = g^{G_{\alpha}}$. Thus, applying \Cref{lem:cgtranss}, we find that
$\Centraliser{G}{g}$ acts transitively on $\Omega_g$ and $|\Omega_g| =
|\Centraliser{G}{g}: \Centraliser{G}{g} \cap G_{\alpha}|$. 
Using \Cref{lem:reecent}, we obtain the desired formulae for
$|\Omega_g|$. 
\end{proof}

\subsection{Ree Groups Acting on Generalised Quadrangles}
We now prove \Cref{thm:main} in the case where the socle is isomorphic to 
$\Ree{q}$. 
\begin{proof}
Suppose that an almost simple group with socle $\Ree{q}$ acts primitively on 
the point-set and the line-set of a generalised quadrangle $\Gamma = 
(\cP,\cL,\cI)$ with order $(s,t)$ where $s,t \geqslant 2$. 
Then $|\cP| = (s+1)(st+1)$. Note that $s+1 \geqslant 3$ and 
$st+1 \geqslant 3$. Therefore, their product cannot be a prime. The argument 
in \cite[Section 4]{morgan} shows that the socle of this almost simple group 
acts primitively on both $\cP$ and $\cL$. Therefore, it suffices to consider 
the Ree group $\Ree{q}$ acting primitively on both $\cP$ and $\cL$. Let 
$G = \Ree{q}$. For a point $\alpha \in \cP$ and a line $L \in 
\cL$, their respective stabilisers $G_{\alpha}$ and $G_L$ are maximal 
subgroups of $G$. 
We use the list of maximal subgroups from \Cref{thm:reemaxs}. \\

If $G_{\alpha} \cong E_q.E_q.E_q.C_{q-1}$, then the action of $G$ is 
$2$-transitive \cite[Theorem (v)]{ward}. Thus, a pair of collinear points can
be mapped to a pair of non-collinear points, which is a contradiction.  
Therefore, we only need to investigate the cases: $C_2 \times \PSLn{2}{q}$, 
$(E_4 \times D_{(q+1)/4}):C_3$, $C_{q\pm\sqrt{3q}+1}:C_6$ and $\Ree{q_0}$ for 
the point and line-stabiliser and show
that there are no generalised quadrangles in these scenarios. To this end,
suppose $G_{\alpha}$ is isomorphic to $C_2 \times \PSLn{2}{q}$, $(E_4 \times
D_{(q+1)/4}):C_3$, $C_{q\pm\sqrt{2q}+1}:C_4$ or
$\Ree{q_0}$ and $G_{L}$ is isomorphic to $C_2 \times \PSLn{2}{q}$, 
$(E_4 \times D_{(q+1)/4}):C_3$, $C_{q\pm\sqrt{2q}+1}:C_4$ or $\Ree{q_1}$, 
where $q_i = 3^{n_i} > 3$ and $q_i^{r_i} = q$ for primes $r_i$ and $i = 1,2$. 
By \Cref{lem:reeccls}, 
both $G_{\alpha}$ and $G_L$ contain a non-real element of order
$3$, say $g \in G_{\alpha}$ and $h \in G_L$. Moreover, there are 
precisely two classes of non-real elements of order $3$ in $G$, and so 
$g$ is conjugate to 
$h$ or $h^{-1}$. Without loss of generality, suppose $g$ is conjugate to $h$,
and so, we can write $g = h^k$ for some $k \in G$. Hence, $g \in G_L^k$. 
Since $G_L^k = G_{L^{\prime}}$ where 
$L^{\prime} = L^k \in \cL$, we can take $L$ to be $L^{\prime}$ and assume that 
we have an order $3$ non-real element $g \in G_{\alpha} \cap G_L$. Let 
$\cP_g$ and $\cL_g$ be the set of points and the set of lines fixed by $g$, 
respectively. By \Cref{lem:reefixpts}, we find that
the number of fixed points and lines are 
\[
|\cP_g| = 3^a \ \textrm{and} \ |\cL_g| = 3^b,
\]
where $a \in A \defi \{2m+1,4m+1,2n_0(r_0-1)\}$ and $b \in
B \defi \{2m+1,4m+1,2n_1(r_1-1)\}$. 
By \Cref{lem:suzsub}, the
fixed substructure, $(\cP_g,\cL_g,\cI_g)$ is a generalised quadrangle of
order $(s^{\prime},t^{\prime})$ for some positive integers $s^{\prime}$ and
$t^{\prime}$. Therefore, 
\[
|\cP_g| = (s^{\prime}+1)(s^{\prime}t^{\prime}+1) = 3^a \ \textrm{and} \ 
|\cL_g| = (t^{\prime}+1)(s^{\prime}t^{\prime}+1) = 3^b.
\]
However, there are no solutions for $s^{\prime}$ and $t^{\prime}$ by
\Cref{lem:suzupts} and so, we have a contradiction. Therefore, the small 
Ree group $\Ree{q}$ cannot act primitively on both the set of points and 
the set of 
lines of a generalised quadrangle. Consequently, an almost simple group with
socle isomorphic to $\Ree{q}$ cannot act primitively on both the set of points
and the set of lines of a generalised quadrangle. 
\end{proof}



\subsection*{Acknowledgements}
Vishnuram Arumugam received support from the Australian Research Training Program.
The authors thank an anonymous referee for suggesting a slicker proof of \Cref{lem:suzupts}.

\hypertarget{ref}{}
\printbibliography[heading=bibintoc, title={References}]

@article{glasbyoct,
author = {Glasby, Stephen and Pierro, Emilio and Praeger, Cheryl},
year = {2021},
month = {04},
pages = {},
title = {Point-primitive generalised hexagons and octagons and projective linear groups},
volume = {21},
journal = {Ars Mathematica Contemporanea},
doi = {10.26493/1855-3974.2049.3db}
}

@book{maldy,
  title={Generalized Polygons},
  author={Van Maldeghem, H.},
  isbn={9783034802703},
  lccn={2011941498},
  series={Modern Birkh{\"a}user Classics},
  year={2012},
  publisher={Springer Basel}
}

@book{kiss,
  title={Finite Geometries},
  author={Kiss, G. and Szonyi, T.},
  isbn={9781351646383},
  year={2019},
  publisher={CRC Press}
}

@book{payne,
  title={Finite Generalized Quadrangles},
  author={Payne, S.E. and Thas, J.A. and European Mathematical Society},
  isbn={9783037190661},
  series={EMS series of lectures in mathematics},
  year={2009},
  publisher={European Mathematical Society}
}

@article{bamberg,
title = {Point-primitive generalised hexagons and octagons},
journal = {Journal of Combinatorial Theory, Series A},
volume = {147},
pages = {186-204},
year = {2017},
doi = {10.1016/j.jcta.2016.11.008},
author = {John Bamberg and S.P. Glasby and Tomasz Popiel and Cheryl E. Praeger and Csaba Schneider},
}

@article{luzhangzou,
title={Nonexistence of generalized quadrangles admitting a point-primitive and line-primitive automorphism group with socle ${\rm PSU}(3,q)$, $q\geq 3$}, 
journal={Journal of Algebraic Combinatorics},
volume={60},
pages={871-898},
author={Jianbing Lu and Yingnan Zhang and Hanlin Zou},
year={2024},
doi={10.1007/s10801-024-01355-6}
}

@article{schneider,
title = {Primitive flag-transitive generalized hexagons and octagons},
journal = {Journal of Combinatorial Theory, Series A},
volume = {115},
number = {8},
pages = {1436-1455},
year = {2008},
issn = {0097-3165},
doi = {10.1016/j.jcta.2008.02.004},
author = {Csaba Schneider and Hendrik {Van Maldeghem}},
}

@article{tits,
     author = {Tits, Jacques},
     title = {Sur la trialit\'e et certains groupes qui s'en d\'eduisent},
     journal = {Publications Math\'ematiques de l'IH\'ES},
     pages = {13--60},
     publisher = {Institut des Hautes \'Etudes Scientifiques},
     volume = {2},
     year = {1959},
}

@article{feit,
title = {The nonexistence of certain generalized polygons},
journal = {Journal of Algebra},
volume = {1},
number = {2},
pages = {114-131},
year = {1964},
doi = {10.1016/0021-8693(64)90028-6},
author = {Walter Feit and Graham Higman}
}

@article{buek,
title = {Finite distance-transitive generalized polygons},
journal = {Geometriae Dedicata},
volume = {52},
pages = {41-51},
year = {1994},
doi = {10.1007/BF01263523},
author = {Francis Buekenhout and Hendrik Van Maldeghem}
}

@article{bambergevans,
title = {No sporadic almost simple group acts primitively on the points of a generalised quadrangle},
journal = {Discrete Mathematics},
volume = {344},
number = {4},
pages = {112291},
year = {2021},
doi = {10.1016/j.disc.2021.112291},
author = {John Bamberg and James Evans},
}

@article{bgmrs,
title = "Generalised quadrangles with a group of automorphisms acting primitively on points and lines",
author = "John Bamberg and Michael Giudici and J. Morris and Gordon Royle and Pablo Spiga",
year = "2012",
doi = "10.1016/j.jcta.2012.04.005",
language = "English",
volume = "119",
pages = "1479--1499",
journal = "Journal of Combinatorial Theory Series A",
publisher = "Academic Press",
number = "7",
}

@article{fixpts,
    author = {Liebeck, Martin W. and Saxl, Jan},
    title = "{Minimal Degrees of Primitive Permutation Groups, with an Application to Monodromy Groups of Covers of Riemann Surfaces}",
    journal = {Proceedings of the London Mathematical Society},
    volume = {s3-63},
    number = {2},
    pages = {266-314},
    year = {1991},
    month = {09},
    doi = {10.1112/plms/s3-63.2.266},
}

@MISC {psl,
%    TITLE = {The number of conjugacy classes of the simple group PSL(2,q)},
%    AUTHOR = {Michael Zieve},
%    HOWPUBLISHED = {MathOverflow},
%    URL = {https://mathoverflow.net/q/142691}
%}

@article{suzuki,
 doi = {10.2307/1970423},
 author = {Michio Suzuki},
 journal = {Annals of Mathematics},
 number = {1},
 pages = {105--145},
 publisher = {Annals of Mathematics},
 title = {On a Class of Doubly Transitive Groups},
 volume = {75},
 year = {1962}
}

@article{fenglu,
author = {Feng, Tao and Lu, Jianbing},
title = {On finite generalized quadrangles with PSL(2,q) as an automorphism group},
year = {2023},
issue_date = {Jun 2023},
publisher = {Kluwer Academic Publishers},
address = {USA},
volume = {91},
number = {6},
doi = {10.1007/s10623-023-01203-x},
journal = {Des. Codes Cryptography},
month = {3},
pages = {2347–2364},
numpages = {18},
}

@article{morgan,
title = "Generalised Polygons Admitting a Point-Primitive Almost Simple Group of Suzuki or Ree Type",
author = "Luke Morgan and Tomasz Popiel",
year = "2016",
month = "2",
volume = "23",
journal = "Electronic Journal of Combinatorics",
note = {{\#}P1.34}
}

@book{maxlowdim, 
place={Cambridge}, 
series={London Mathematical Society Lecture Note Series}, 
title={The Maximal Subgroups of the Low-Dimensional Finite Classical Groups}, 
DOI={10.1017/CBO9781139192576}, 
publisher={Cambridge University Press}, 
author={Bray, John N. and Holt, Derek F. and Roney-Dougal, Colva M.}, 
year={2013}, 
collection={London Mathematical Society Lecture Note Series}
}

@book{isaacsfin,
  title={Finite Group Theory},
  author={Isaacs, I.M.},
  isbn={9780821843444},
  series={Graduate studies in mathematics},
  year={2008},
  publisher={American Mathematical Society}
}

@article{ward,
 author = {Harold N. Ward},
 journal = {Transactions of the American Mathematical Society},
 number = {1},
 pages = {62--89},
 publisher = {American Mathematical Society},
 title = {On Ree's Series of Simple Groups},
 volume = {121},
 year = {1966}
}
\end{document}